\documentclass[12pt]{elsevier/elsart}
\usepackage{graphics}
\usepackage{graphicx}
\usepackage[dvips]{epsfig}
\usepackage{subfigure}
\usepackage{color}
\usepackage{amssymb,amsmath,amscd}
\usepackage{amsfonts}

\newcommand{\C}{\mathbb{C}}
\newcommand{\R}{\mathbb{R}}
\newcommand{\Z}{\mathbb{Z}}

\newenvironment{proof}{\noindent{\it Proof}\rm.}{\hfill $\Box$}

\theoremstyle{plain}
\newtheorem{theorem}{Theorem}[section]
\newtheorem{lemma}{Lemma} [section]

\newtheorem{exs}[thm]{Examples}
\newtheorem{proposition}[theorem]{Proposition}

\theoremstyle{definition}
\newtheorem{remarks}[theorem]{Remarks}
\newtheorem{remark}[theorem]{Remark}
\numberwithin{equation}{section}
\theoremstyle{definition}
\newtheorem{definition}{Definition}[section]
\let\frak=\mathfrak

\begin{document}
\begin{frontmatter}

\title{On sums of graph eigenvalues}
\author{Evans M. Harrell II}
\ead{harrell@math.gatech.edu}
\address{School of Mathematics, Georgia Institute of
Technology, Atlanta, GA 30332-0160 U.S.A.}
\author{Joachim Stubbe}
\ead{Joachim.Stubbe@epfl.ch}
\address{
EPFL, MATH-GEOM,
Station 8,
CH-1015 Lausanne, Switzerland
}

\maketitle

\maketitle
\begin{abstract}
We use two variational techniques
to prove upper bounds for sums of the lowest several eigenvalues of matrices
associated with finite, simple, combinatorial graphs.
These include estimates for the adjacency matrix of a graph and for
both the standard combinatorial Laplacian and the renormalized Laplacian.  We also provide upper bounds for sums of squares of eigenvalues of these three matrices.

Among our results, we generalize an inequality of Fiedler for the extreme eigenvalues of the
graph Laplacian to a bound on the sums of the smallest (or largest) $k$ such eigenvalues,
$k < n$.

Furthermore, if $\lambda_j$ are the eigenvalues of the graph Laplacian $H = - \Delta$, in increasing order,
on a finite graph with $|\mathcal{V}|$ vertices and $|\mathcal{E}|$ edges which
is isomorphic to a subgraph
of the $\nu$-dimensional infinite cubic lattice, then the spectral sums obey a
Weyl-type upper bound, a simplification of which reads

\begin{equation}\nonumber
\sum_{j=1}^{k-1}{\lambda_j} \le \frac{\pi^2 |\mathcal{E}|}{3} \left(\frac{k}{|\mathcal{V}|} \right)^{1+\frac{2}{\nu}}
\end{equation}
for each $k < |\mathcal{V}|$.

This and related estimates for $\sum_{j=1}^{k-1}{\lambda_j}^2$ provide
a family of necessary conditions for the
embeddability of the graph in a lattice of dimension $\nu$ or less.

\end{abstract}
\end{frontmatter}

\section{Introduction}

It is possible to discern some
structural features of a graph $G$ from the spectra of various matrices associated with $G$.
In practice, the most important such matrices are the adjacency matrix,
the graph Laplacian, and the renormalized Laplacian favored for example by
Chung \cite{Chu}.  (In this
article, a graph will be assumed to be finite, simple,
connected, and non-directed without further comment.  For the
definitions of these terms and other general theory,
we refer to \cite{Bol,Die})
The eigenvalues of these three matrices have been well studied
and are discussed in several monographs, especially
\cite{Chu,CvRoSi1,CvRoSi2,BiLeSt}.
The particular objects of the present study are the
(incomplete) sums of the ordered eigenvalues associated with a graph,
and related quantities such as sums of powers of eigenvalues.
We use two variational methods to obtain upper bounds on
the partial sums of eigenvalues, which 
reflect the topology of the graph and the possibility of embedding it in 
a regular lattice.  
The inequalities are for
the most part optimal in the sense that,
given a little information about the structure of the graph,
there are examples in which the inequalities are saturated.

There is a long history
in quantum physics
(e.g., \cite{LiTh,ReSiIV})
and in spectral geometry (e.g., \cite{Pol,LiYa}) 
of investigation of the
sums of the lowest $k$ eigenvalues of operators,
and relating them to the nature of the phase space or to the geometry,
but eigenvalue sums have received much less attention in the context of graph spectra.
In the main, the {\em complete} sums of eigenvalues have been
recognized as a kind of energy connected to the
structure of graphs and have been studied, for example
in \cite{Zho}.

The notational conventions of the standard references on
graph spectra are,
unfortunately, not consistent with one another.  Because of this we recall
some basic definitions to fix the notation to be used.

\begin{definition}
Given 
a graph $G$ with
$|\mathcal{V}| = n$ vertices, labeled in some fashion, the {\em adjacency matrix} $A$ has elements $a_{uv} = 1$ when vertex $u$ is connected to vertex $v$ and $0$ otherwise.  The {\em (combinatorial) graph Laplacian} is defined as
$$
\Delta := A - Deg,
$$
where $Deg$ is the diagonal matrix such that $Deg_{vv} = d_v$ is the degree of the vertex $v$, i.e., the number of edges connecting to $v$.  We prefer to express
our results in terms of $H := - \Delta$, noting that $H$ is positive
semidefinite, since
\begin{equation}\label{weakLapDef}
\left\langle{\phi, H \phi}\right\rangle
= \sum_{{\rm edges}\, [uv]}{|\phi_u - \phi_v|^2} = \frac 1 2 \sum_u \sum_v{|\phi_u - \phi_v|^2}.
\end{equation}
The null space of $H$ includes the constant vector with all entries equal to $1$, which we denote
$\bf{1}$, and is one-dimensional
(assuming that the graph is connected).

The {\em renormalized Laplacian}, cf.\ \cite{Chu}, corresponds to the matrix
$$
\hat{H} := Deg^{-\frac 1 2}HDeg^{-\frac 1 2}.
$$

Our notation for the eigenvalues of these three matrices is as follows:
\begin{align}
&\nonumber
A: \quad \alpha_0 > \alpha_1 \ge ...  \alpha_{n-1}\\
&\nonumber
H: \quad 0 = \lambda_0 < \lambda_1 \le \lambda_2 \le ...  \lambda_{n-1}\\
&\nonumber
\hat{H}: \quad 0 = \frak{c}_0 < \frak{c}_1 < \frak{c}_2 \le ...  \frak{c}_{n-1} \le 2.
\end{align}
\end{definition}
The indexing scheme ensures that in the case of a regular graph of degree $d$,
$\lambda_k = d - \alpha_k = d \frak{c}_k$ for each $k$.  When discussing an arbitrary matrix
(usually self-adjoint), we shall call it and its eigenvalues $(M, \mu_0 \dots \mu_{n-1})$.  Throughout the article, $n = |\mathcal{V}|$ is reserved for the number of vertices, and $m = |\mathcal{E}|$ designates
the number of edges.

For later purposes we recall some basic
identities relating the spectra of $H$ and $A$ to properties of the graph:
\begin{equation}\label{edgetrace}
  \text{Tr}\,(H)=\text{Tr}\,(A^2) = \sum_{v}d_v = 2 m, \quad\quad
\text{Tr}\,(H^2)= 2 m +\sum_{v}d_v^2.
\end{equation}
The topological quantity $\sum_{\mathbf{x}}d_\mathbf{x}^2$ is known as
{\em the first Zagreb index} of the graph $G$,
denoted $M_1(G)$ \cite{GuTr}.

\section{An extension of a result of Fiedler to sums of graph eigenvalues}

The usual
variational strategy for estimating the spectrum of an operator is to make
shrewd, simplifying guesses at the eigenvectors, and to use them in inequalities
deriving from the spectral theorem.
In this section we exploit the additive version of the min-max principle for the eigenvalues
$\{\mu_\ell\}_{\ell=0}^{n-1}$
of a
self-adjoint matrix $M$, {\em viz}., that
for any orthonormal set of vectors $\{\mathbf{\phi}^{(\ell)}\}_{\ell=0}^{k-1}$,
\begin{equation}\label{minmax}
\sum_{\ell=0}^{k-1}{\mu_\ell} \le \sum_{\ell=0}^{k-1}{\left\langle M \mathbf{\phi}^{(\ell)}, \mathbf{\phi}^{(\ell)}\right\rangle},
\end{equation}
and
\begin{equation}\label{maxmin}
\sum_{\ell=k}^{n-1}{\mu_\ell} \ge \sum_{\ell=k}^{n-1}{\left\langle M \mathbf{\phi}^{(\ell)}, \mathbf{\phi}^{(\ell)}\right\rangle},
\end{equation}
cf.\ \cite[\S 34]{BeBe}.  With the aid of a special basis we shall obtain sharp bounds on sums of eigenvalues,
which reduce to a result of Fiedler when there is only one summand.  In the following section we
obtain some different, competing results on sums of eigenvalues and related quantities, using a
novel
variational argument that incorporates
spectral projectors and an averaging over a family of test functions.

A good way to come up with test vectors for use in
\eqref{minmax}
is to consider special graphs on $n$ or more
vertices, with eigenvectors that are known explicitly.
Thus we consider
a graph $G_p$ which is the join of a complete graph with a
completely disconnected graph.  That is,
there are $n$ vertices of which the first $n-p$ vertices
have no edges in common, but the graph is otherwise maximally connected.
For $1\leq p \leq n-1$ the graph Laplacian $H_p$ has the form

\begin{equation}\label{H-p-patrix}
    H_{p}:=\left(
            \begin{array}{ccccccccc}
              p      & 0     & \ldots & 0      & \vline & -1 & -1 & \ldots & -1 \\
              0      & p     & \ldots & 0      & \vline & -1 & -1 & \ldots & -1 \\
              \vdots &       &        & \vdots & \vline & \vdots &  &  & \vdots \\
              0      & 0     & \ldots & p      & \vline & -1 & -1 & \ldots & -1 \\
              \hline
              -1     & -1    & \ldots & -1     & \vline & n-1 & -1 & \ldots & -1 \\
              -1     & -1    & \ldots & -1     & \vline & -1 & n-1 &  & -1 \\
              \vdots &\vdots &        & \vdots & \vline & \vdots &  &  & \vdots \\
              -1     & -1    & \ldots & -1     & \vline & -1 & -1 & \ldots & n-1 \\
            \end{array}
          \right).
\end{equation}

\begin{rem}
In particular, $H_1$ is the Laplacian of a star graph, while $H_{n-1}$ is the Laplacian
of a complete graph.
For future purposes we observe that
\begin{equation}\label{traces-H-p}
     \text{tr}(H_p)=p(2n-p-1),\quad \text{tr}(H_p^2)=p(n^2+pn-p^2-p).
\end{equation}
\end{rem}

Building a variational estimate for an arbitrary graph
from the eigenvectors of this family of graphs
leads to an extension of the result of Fiedler \cite{Fie},
as we next demonstrate.  Letting
$\mathbf{e}_j,j=1\ldots,n$ denote the canonical orthonormal basis vectors of $\mathbb{R}^n$, we
construct a {\em reduced basis} of eigenfunctions as follows.

\begin{proposition} [Spectral analysis of $H_p$]
Let $\displaystyle \mathbf{\epsilon}^{(0)}  := \frac{{\bf{1}}}{\sqrt{n}}
:= \frac{1}{\sqrt{n}} \sum_{j=1}^n{\mathbf{e}_j}$ and
for $\ell=1\ldots, n-1$, define
\begin{equation}\label{on-basis}
   \mathbf{\epsilon}^{(\ell)}:=\frac1{\sqrt{\ell(\ell+1)}}\big(\ell\mathbf{e}_{\ell+1}-\sum_{j=1}^{\ell}\mathbf{e}_j\big),
\end{equation}
noting that
$\{\mathbf{\epsilon}^{(\ell)}\}$ is an orthonormal basis of $\mathbb{R}^n$.
For each $\ell=1,\ldots,n-p-1$, $\mathbf{\epsilon}^{(\ell)}$ is an eigenvector of $H_p$ with corresponding eigenvalue $p$, and for each $\ell=n-p,\ldots,n-1$, $\mathbf{\epsilon}^{(\ell)}$ is an eigenvector of $H_p$ with corresponding eigenvalue $n$.  (Trivially, $\mathbf{\epsilon}^{(0)}$ is the normalized eigenvector of $H_p$ with
eigenvalue $0$.)
\end{proposition}

The proposition can be verified directly.   More details about the
spectral analysis of the graphs $G_p$ are collected in Appendix A.

\textbf{Matrix elements of a general graph in the reduced basis $\{\mathbf{\epsilon}^{(\ell)}\}$}.
We first compute the matrix elements of a general self-adjoint matrix
$M$ with respect to
$\mathbf{\epsilon}^{(\ell)}$, $\ell=1,\ldots, n-1$. Let $\displaystyle C_{j,\ell}=\frac1{\sqrt{j(j+1)\ell(\ell+1)}}$. Then
\begin{equation}\label{H-matrix-elements}
\begin{split}
    C_{j,\ell}^{-1}\langle\mathbf{\epsilon}^{(j)},M\mathbf{\epsilon}^{\ell}\rangle&=j \ell\,m_{j+1,\ell+1}+
    \sum_{\alpha=1}^,\sum_{\beta=1}^\ell m_{\alpha\beta}\\
    &\quad -j\sum_{\beta=1}^\ell m_{\beta,j+1}-\ell\sum_{\alpha=1}^jm_{\alpha,\ell+1}.\\
    \end{split}
\end{equation}
In particular,
\begin{equation}\label{H-matrix-diagonal}
   \langle\mathbf{\epsilon}^{(\ell)},M\mathbf{\epsilon}^{(\ell)}\rangle=\frac1{\ell(\ell+1)}\bigg(\ell^2\,m_{\ell+1,\ell+1}+\sum_{\alpha=1}^\ell\sum_{\beta=1}^\ell m_{\alpha\beta}-2\ell \sum_{\alpha=1}^\ell m_{\alpha,\ell+1}\bigg).
\end{equation}
If we now specialize so that $M = H$, a graph Laplacian, then
\begin{equation*}
    \langle\mathbf{\epsilon}^{(1)},H\mathbf{\epsilon}^{(1)}\rangle
    = \frac{1}{2} \left(d_1+d_2\right) - a_{1 2},
\end{equation*}
or, using the fact that the sum over rows of $H$ is equal to zero,
\begin{equation}
    \langle\mathbf{\epsilon}^{(n-2)},H\mathbf{\epsilon}^{(n-2)}\rangle=\frac{(n-1)^2d_{n-1}-2(n-1)a_{n\,n-1}+d_{n}}{(n-2)(n-1)},
\end{equation}
\begin{equation}\label{exval2.9}
    \langle\mathbf{\epsilon}^{(n-1)},H\mathbf{\epsilon}^{(n-1)}\rangle=\frac{n}{n-1}\,d_{n}.
\end{equation}
Similarly, we compute
\begin{equation}
    \langle H\mathbf{\epsilon}^{(n-1)},H\mathbf{\epsilon}^{(n-1)}\rangle
=\frac{n}{n-1}\,\big(d_n^2+d_n\big).
\end{equation}
If the diagonal elements of $H$
are arranged in decreasing order, then
applying the variational principle \eqref{minmax}
to \eqref{exval2.9}
immediately yields an alternative proof of a result of Fiedler \cite{Fie}:
\begin{prop} [Fiedler]  \label{Fiedler}
For the graph Laplacian,
\begin{equation}\label{ritz-bound}
    \lambda_1\leq \frac{n}{n-1}\,\underset{v}{\min}\,d_{v},\quad \frac{n}{n-1}\,\underset{v}{\max}\,d_{v}\leq \lambda_{n-1},
\end{equation}
with equality for the complete graph and the star graph.
\end{prop}

We are now ready to extend
Proposition \ref{Fiedler} to sums of ordered eigenvalues.  Applying the min-max
principle for sums of eigenvalues \eqref{minmax}, choosing $\mathbf{\phi}^{({\ell})} = \mathbf{\epsilon}^{(\ell)}$
and using the fact that we may relabel vertices, we get the following.
\begin{prop} \label{Fiedler-extended}
The partial sums of the
eigenvalues of the graph Laplacian satisfy the
following inequalities.
\begin{equation}\label{ritz-bound-sums}
\begin{split}
    &\lambda_1+\lambda_2\leq \frac{n-1}{n-2}\,\underset{u\neq v}{\min}\,\big(d_u+d_v
    -\frac{2a_{uv}}{n-1}\big),\\
 &\frac{n-1}{n-2}\,\underset{u \neq v}{\max}\,\big(d_u+d_v-\frac{2a_{uv}}{n-1}\big)\leq \lambda_{n-2}+\lambda_{n-1}\\
\end{split}
\end{equation}
with equality for the complete graph and the star graph.  Moreover
(by averaging over $u$),
\begin{equation}\label{ritz-bound-sums2}
\begin{split}
    &\lambda_1+\lambda_2\leq \frac{2 m}{n-2}+ \frac{n(n-3)}{(n-1)(n-2)}\,\underset{v}{\min}\,d_v,\\
 &\frac{2 m}{n-2}+ \frac{n(n-3)}{(n-1)(n-2)}\,\underset{v}{\max}\,d_v \leq \lambda_{n-2}+\lambda_{n-1}.\\
\end{split}
\end{equation}
For any $L=1,\ldots,n-1$ we get
\begin{equation}\label{ritz-bound-sums-L}
\begin{split}
    &\sum_{i=1}^L\lambda_i\leq \frac{L}{L+1}\;\sum_{x=1}^{L+1}d_x+\frac{1}{L+1}\,\;\underset{v \neq u}{\sum_{u=1}^{L+1}\sum_{v =1}^{L+1}}a_{u v }\leq \sum_{j=n-L+1}^n\lambda_j\\
 &\sum_{i=1}^L\lambda_i\leq \frac{n-L+1}{n-L}\;\sum_{x=n-L+1}^{n}d_x+\frac{1}{n-L}\,\;\underset{v \neq u}{\sum_{u=n-L+1}^{n}\sum_{v =n-L+1}^{n}}a_{u v }\leq \sum_{j=n-L+1}^n\lambda_j.\\
\end{split}
\end{equation}
\end{prop}

\section{An averaged variational principle and consequences for spectral sums}

In \cite{Kro} P. Kr\"oger
proved an upper bound for sums of eigenvalues of a vibrating free membrane (i.e. the Neumann Laplacian) on a
bounded domain.  Kr\"oger's
bound is sharp in the sense of 
having the same dependence on dimension as the classic asymptotic estimate 
of large eigenvalues 
due to Weyl.
Although as presented in \cite{Kro} the bound appears to rely on special properties of the Laplacian and of the Fourier transform, in our view the essence of the argument was that it averaged different parts
of a variational estimate in different ways, one of which simplified some coefficients.  We shall formulate an abstract version of the spectral estimate of \cite{Kro}
and apply it to two situations, in one of which the graph
is assumed to be a finite subset of the lattice $\Z^\nu$ equipped with nearest-neighbor edges,
which we term the
{\em cubic lattice graph} $\frak{Q}^\nu$.
Under this assumption, we obtain an analogue for graphs
of what Kr\"oger proved for the Neumann problem on a compact $\Omega \subset \R^\nu$,
and in particular we obtain an upper bound with
Weyl dependence on dimension.  The second situation is more
generic, and applies to an arbitrary graph
on $n$ vertices.

Suppose that $M$ is a self-adjoint operator on a Hilbert space $\mathcal{H}$,
with discrete eigenvalues
$-\infty < \mu_0 \le \mu_1 \le \dots$.
Let $P_k$ be the spectral projector associated to the eigenvalues $0$ through $k$, and let
$f$ be in the quadratic-form domain
$\mathcal{Q}(M) \subset \mathcal{H}$.
(We reassure the reader that in this article all operators will be bounded matrices,
in which case 
domain technicalities are avoided, as
$\mathcal{Q}(M)$ coincides with $\mathcal{H}$, and indeed
$\mathcal{H}$ will merely
be $\C^n$.)

By the variational principle \eqref{minmax},
\begin{equation}\label{var-principle}
 \mu_{k}\big(\langle f,f\rangle-\langle P_{k-1} f, P_{k-1} f\rangle\big)\leq \langle M f,f \rangle-\langle MP_{k-1} f, P_{k-1} f\rangle.
\end{equation}
Now consider a family of such trial functions $f_z$ indexed by
a variable over which we can average.  By averaging over
two different sets, we get the following variational principle,
corresponding to the main theorem of \cite{Kro}.
\begin{theorem}\label{Krlemma}
Consider a self-adjoint operator $M$ on a Hilbert space $\mathcal{H}$,
with ordered, entirely discrete
spectrum $- \infty < \mu_0 \le \mu_1 \le \dots$
and corresponding normalized eigenvectors $\{\mathbf{\psi}^{(\ell})\}$.
Let $f_z$ be
a family of vectors in $\mathcal{Q}(M)$ indexed by
a variable $z$ ranging over
a measure space $(\mathfrak{M},\Sigma,\sigma)$.
Suppose that $\mathfrak{M}_0$ is a
subset of $\mathfrak{M}$.  Then for any eigenvalue $\mu_k$ of $M$,
\begin{equation}\label{kroeger-2}
\begin{split}
    &\mu_{k} \bigg(\int_{\mathfrak{M}_0}\langle f_z,f_z\rangle\,d \sigma-\sum_{j=0}^{k-1}\int_{\mathfrak{M}}|\langle f_z,\mathbf{\psi}^{(j)}\rangle|^2\,d \sigma\bigg)\\
&\leq\\
&\int_{\mathfrak{M}_0}{\langle H f_z,f_z\rangle d \sigma}-\sum_{j=0}^{k-1}\mu_j\int_{\mathfrak{M}}|\langle f_z,\mathbf{\psi}^{(j)}\rangle|^2\,d \sigma,\\
\end{split}
\end{equation}
provided that the integrals converge.
\end{theorem}

\begin{proof}
By integrating \eqref{var-principle},
\begin{align}
& \mu_{k}\int_{\mathfrak{M}_0}\left(\langle f_z,f_z\rangle-\langle P_{k-1} f, P_{k-1} f_z\rangle\right)
d \sigma\\
&\quad  \leq \int_{\mathfrak{M}_0}\langle M f_z,f_z \rangle \,d \sigma-\int_{\mathfrak{M}_0}\langle MP_{k-1} f_z, P_{k-1} f_z \rangle\,d \sigma,\nonumber
\end{align}
or
\begin{align}\label{averaged-var-principle}
& \mu_{k}\int_{\mathfrak{M}_0}\left(\langle f_z,f_z\rangle-\sum_{j=0}^{k-1}|\langle f_z,
\mathbf{\psi}^{(j)}\rangle|^2\right)d \sigma\\
&\quad \leq \int_{\mathfrak{M}_0}\langle M f_z,f_z \rangle \,d \sigma-\int_{\mathfrak{M}_0}\sum_{j=0}^{k-1}\mu_j|\langle f_z,\mathbf{\psi}^{(j)}\rangle|^2\,d \sigma.\nonumber
\end{align}
Since $\mu_{k}$ is larger than or equal to
any weighted average of $\mu_{1} \dots \mu_{k-1}$,
we add to \eqref{averaged-var-principle} the inequality
\begin{equation}
- \mu_{k}\int_{\mathfrak{M}\setminus \mathfrak{M}_0}\left(\sum_{j=0}^{k-1}|\langle f_z,\mathbf{\psi}^{(j)}\rangle|^2\right)d \sigma
 \leq - \int_{\mathfrak{M} \setminus \mathfrak{M}_0}\sum_{j=0}^{k-1}\mu_j|\langle f_z,\mathbf{\psi}^{(j)}\rangle|^2\,d \sigma,
\end{equation}
and obtain the claim.
\end{proof}

Although Theorem \ref{Krlemma} appears designed to bound $\mu_{k}$,
its most notable use is to provide an upper bound on
$\mu_0 + \dots + \mu_{k-1}$ by
arranging that the left side be nonnegative, under which condition
\begin{equation}\label{KrIneq}
\sum_{j=0}^{k-1}\mu_j\int_{\mathfrak{M}}|\langle f_z,\mathbf{\psi}^{(j)}\rangle|^2\,d \sigma
\le \int_{\mathfrak{M}_0}{\langle M f_z,f_z\rangle d \sigma}.
\end{equation}
In this work, inequalities obtained using Theorem \ref{Krlemma}
will turn out to satisfy the hypotheses of a
celebrated theorem of J. Karamata 
(e.g., see \cite[\S 28]{BeBe}),
which we restate here in a slightly extended version:

\begin{lemma}[{Karamata-Ostrowski}] \label{Karamata}
Let two nondecreasing ordered sequences of real numbers $\{\mu_j\}$ and $\{m_j\}$, $j = 0, \dots, n-1$, satisfy
\begin{equation}\label{KaraCond}
\sum_{j=0}^{k-1}{\mu_j} \le \sum_{j=0}^{k-1}{m_j}
\end{equation}
for each $k$.
Then for any differentiable convex function $\Psi(x)$,
$$
\sum_{j=0}^{k-1}{\Psi(\mu_j)} \ge \sum_{j=0}^{k-1}{\Psi(m_j)}+\Psi'(m_{k-1})\cdot\sum_{j=0}^{k-1}({\mu_j} - {m_j}).
$$
In particular,
assuming either that $\Psi$ is nonincreasing
or that $\displaystyle \sum_{j=0}^{k-1}{\mu_j} =\sum_{j=0}^{k-1}{m_j}$,
$$
\sum_{j=0}^{k-1}{\Psi(\mu_j)} \ge \sum_{j=0}^{k-1}{\Psi(m_j)}
$$
for each $k$.  Similarly,
for any nondecreasing concave
function $\Phi(x)$ and each $k$,
$$
\sum_{j=0}^{k-1}{\Phi(\mu_j)} \le \sum_{j=0}^{k-1}{\Phi(m_j)}.
$$
\end{lemma}

As our first application of Theorem \ref{Krlemma}, using the same stratagem as in
\cite{Kro},
we suppose that a graph $G$ is a finite subgraph of the
cubic lattice graph $\frak{Q}^\nu$.
By definition $\frak{Q}^\nu$ has countably many vertices, which
will be labeled by integer-valued vectors ${\bf x} \in \Z^\nu$ rather than by a single integer
$v$. Two vertices are adjacent
precisely when all but one of the coordinates of ${\bf x}_1$ and
${\bf x}_2$ are equal, while the remaining coordinate differs by $\pm 1$.
As a finite subgraph, the vertex and edge sets of $G$ are subsets of those of
$\frak{Q}^\nu$.

\begin{prop}\label{lattice ex}
Suppose that $G$ is a finite subgraph of $\frak{Q}^\nu$.  Then for $k \ge 1$
the eigenvalues of the graph Laplacian $H_G$ satisfy
\begin{equation}\label{latticeWeyl}
\sum_{j=0}^{k-1}{\lambda_j} \le 2 m \kappa \left(1 - {\rm sinc}(\kappa ^{1/\nu} \pi)\right),
\end{equation}
where ${\rm sinc}(x) := \sin(x)/x$, and $\kappa := k/n$.
Moreover, the sum of squares can be bounded in terms of
simple topological properties of the graph
including
the number of pairs of neighbors of vertex ${\bf p}$ that are collinear, which we denote $d_{\bf p}^\|$:
\begin{equation}\label{latticeWeyl2}
\begin{split}
  \sum_{j=0}^{k-1}{\lambda_j^2} &\leq \, \kappa \,(1-\rm sinc (\pi\kappa^{1/\nu}))^2 \text{Tr}\, (H^2)\\
   &\quad + 2 \,\kappa\, \rm sinc (\pi\kappa^{1/\nu})(1-\, \rm sinc (\pi\kappa^{1/\nu})) \sum_{{\bf x} \in G} d_{\bf x}\\
  &\quad - 2\,\kappa\, \rm sinc (\pi\kappa^{1/\nu})(1-\cos (\pi\kappa^{1/\nu}))\sum_{{\bf x} \in G} d^\|_{\bf x}.\\
\end{split}
\end{equation}
\end{prop}

\begin{remarks}
1. Because of \eqref{edgetrace}, Inequalities \eqref{latticeWeyl} and \eqref{latticeWeyl2}
show that four topological properties of the graph, {\it viz}., the dimension of the ambient lattice,
the number of edges of $G$, its Zagreb index, and the quantity $\sum_{{\bf x} \in G} d^\|_{\bf x}$, control the distribution of eigenvalues of subgraphs of a cubic lattice.  The upper bounds are increasing functions of the dimension, 
which means that these estimates provide a family of necessary conditions for embeddability of the graph in a lattice of dimension $\nu$ or less.  The authors plan to discuss further
spectral conditions for embeddability
of graphs in regular lattices in a future article.

2. As a simplification of \eqref{latticeWeyl}
it is true independently of dimension that
\begin{equation}\label{simplattbd}
\sum_{j=1}^{k-1}{\lambda_j} \le 2 m \kappa,
\end{equation}
which becomes a standard equality when $\kappa= 1$ (i.e., $k=n$).
Inequality \eqref{latticeWeyl2} also yields an equality when $\kappa= 1$.
Another upper bound,
\begin{equation}\label{Weylupper}
\sum_{j=1}^{k-1}{\lambda_j} \le \frac{\pi^2 m}{3} \kappa^{1+\frac{2}{\nu}},
\end{equation}
which has the
form of the Weyl law for Laplacians on
domains $\Omega \subset \R^\nu$, is both better
when $k \ll n$ and
correct to leading order in $\kappa$.

3. Complementary lower bounds for $\sum_{j=k}^{n-1}{\lambda_j}$ are available
as usual by calling upon
$\sum_{j=1}^{n-1}{\lambda_j} = 2 m$ or by passing from $H$ to $-H$. There
are similar bounds when the complementary graph $G'$ is embedded in 
$\mathfrak{Q}^\nu$ owing to the standard relation 
among the nontrivial eigenvalues,
$\lambda_j=n-\lambda'_j$.

4.
When a graph $G$ is
embedded in $\mathfrak{Q}^\nu$ its
Laplacian energy (see \cite{GuZh})
also satisfies a Weyl-type estimate.  Recall that
by definition,
$\displaystyle LE(G):=\sum_{i=0}^{n-1}\left|\lambda_i-\frac{2m}{n}\right|
=2\sum_{i=0}^{n-1}\left(\frac{2m}{n}-\lambda_i\right)_{+}$. Since the variational inequality  \eqref{KrGraphLaplacian} holds when $\lambda_k$ is replaced by $z\in[\lambda_{k-1},\lambda_k]$ it is equivalent to the following inequality for the {\em Riesz mean} of the spectrum,
\begin{equation}\label{KrGraphLaplacianRieszMean}
  \sum_{j}(z-\lambda_j)_{+}\geq zna^{\nu}-2 m a^\nu \left(1 - {\rm sinc}(\pi a)\right)
\end{equation}
for all $z\in[0,2m]$.
After simplifying and optimizing with respect to $a\in[0,1]$,
\eqref{KrGraphLaplacianRieszMean} becomes
\begin{equation}\label{GraphLaplacianRieszMeanWeylLower}
  \sum_{j}(z-\lambda_j)_{+}\geq \frac{2}{\nu}\cdot\frac{m\pi^2}{3}\bigg(\frac{\nu}{\nu+2}\cdot\frac{3nz}{m\pi^2}\bigg)^{1+\frac{\nu}{2}}.
\end{equation}
In particular, the Laplacian energy of a finite subgraph of
$\mathfrak{Q}^\nu$ satisfies
\begin{equation}\label{LEGestimate}
LE(G) \geq 4 m \underset{0\leq a\leq 1}{\max}a^\nu {\rm sinc}(\pi a) \geq \frac{8 m}{\nu +2} \left(\frac{6 \nu}{\pi^2(\nu+2)}\right)^{\nu/2}.
\end{equation}
\end{remarks}

\begin{proof}
We use Theorem \ref{Krlemma}, taking $\mathfrak{M}$ as the cube
$[-\pi, \pi]^\nu$, with Lebesgue measure, and
a vector-valued ${\bf z} \in \mathfrak{M}$; we can then
make the same choice of test functions
as in \cite{Kro}, {\it viz}., $f_{\bf z}({\bf x}) = \exp(i {\bf x}\cdot {\bf z})$.
(However, now think of
$f_{\bf z}$ as a function on the graph $G$ consisting of set of vertices ${\bf x}$ having integer
coordinates,
with a parameter ${\bf z}$ ranging over $\R^\nu$.)
The discrete Fourier transform on functions on $G$ is normalized as
$$
{\hat \phi}({\bf z}) := \sum_{{\bf x} \in G}{\e^{-i {\bf x}\cdot{\bf z}} \phi_{\bf x}},
$$
and we observe that
the  inversion formula for functions in the range of this transform is
\begin{equation}\label{comprel}
\phi_{\bf x} = \frac{1}{(2 \pi)^\nu} \int_{[-\pi,\pi]^\nu}{\e^{i {\bf x}\cdot{\bf z}} {\hat \phi}({\bf z})}.
\end{equation}

We begin by calculating from \eqref{weakLapDef}
\begin{equation}\label{A_0lattice2}
\left\langle H \e^{i {\bf x}\cdot {\bf z}}, \e^{i {\bf x}\cdot {\bf z}}\right\rangle =
\frac 1 2 \sum_{{\bf x} \in G}{\sum_{{\bf q} \sim {\bf x}} {|\e^{i {\bf x}\cdot {\bf z}} - \e^{i {\bf q}\cdot {\bf z}}|^2}}.
\end{equation}
Now, each term ${|\e^{i {\bf x}\cdot {\bf z}} - \e^{i {\bf q}\cdot {\bf z}}|^2}$
simplifies to $|e^{\pm i z_\ell} - 1|^2 = 4 \sin^2\left(\frac{z_\ell}{2}\right)$ for one of the
Cartesian coordinates $z_\ell$.  If we integrate over a cube of the form
$\mathfrak{M}_0 := [-a \pi, a \pi]^\nu$,
then these terms are replaced by $4 (a \pi - \sin(a \pi))$,
and thus the quantity in \eqref{A_0lattice2} evaluates to
$$
2 (a \pi - \sin(a \pi))(2 a \pi)^{\nu-1} \sum_{{\bf x} \in G}{d_{\bf x}}
= (2 a \pi)^\nu 2 \left(1 - \frac{\sin(a \pi)}{a \pi}\right)m,
$$
drawing upon \eqref{edgetrace}.
To evaluate the other quantity on the right side of
\eqref{kroeger-2}, we note that by the completeness relation associated with
\eqref{comprel},
$$
\int_{[-\pi, \pi]^\nu}{|\left\langle{\e^{i {\bf x}\dot{\bf z}}, \mathbf{\psi}^{(j)}}\right\rangle|^2}
= (2 \pi)^\nu \|\mathbf{\psi}^{(j)}\|^2 = (2 \pi)^\nu.
$$
Meanwhile, after integration, the lesser side of \eqref{kroeger-2} becomes
$$
\lambda_{k} \left(n (2 a \pi)^\nu - k(2 \pi)^\nu\right),
$$
and therefore, after division by $(2 \pi)^\nu$, we obtain
\begin{equation}\label{KrGraphLaplacian}
\lambda_{k} \left(n a^\nu - k\right)\leq 2 m a^\nu \left(1 - {\rm sinc}(\pi a)\right)-\sum_{j=0}^{k-1}{\lambda_j}
\end{equation}
for all $0\leq a\leq 1$.
Letting $a^\nu \to \kappa$, we obtain \eqref{latticeWeyl}.

For the inequality on sums of squares we observe that
$$
H f_{\bf z}\big|_{\bf x} =  d_{\bf x} e^{i {\bf x} \cdot {\bf z}} - \sum_{{\bf q}\sim{\bf x}}{e^{i {\bf q} \cdot {\bf z}}}
= e^{i {\bf x} \cdot {\bf z}}\left(d_{\bf x} -  \sum_{{\bf q}\sim{\bf x}}{e^{\pm i z_\ell}}\right),
$$
where as before the Cartesian direction $\ell$ and the sign depend on ${\bf x}$ and ${\bf q}$.  Thus
\begin{align}\label{Hf^2}
\left\langle{f_{\bf z}, H^2 f_{\bf z}}\right\rangle &= \|H f_{\bf z}\|^2\nonumber \\
&=
\sum_{\bf x}{\left(d_{\bf x}^2 - 2 d_{\bf x}
\Re\left(\sum_{{\bf q}\sim{\bf x}}{e^{\pm i z_\ell}}\right) + \left|\sum_{{\bf q}\sim{\bf x}}{e^{\pm i z_\ell}}\right|^2
\right)}.
\end{align}

When integrated in ${\bf z}$ over the cube $[-a \pi, a \pi]^\nu$, the first two contributions to this equation become
$((2 a \pi)^\nu - 4 ((2 a \pi)^{\nu-1} \sin a \pi) \sum_{\bf x}{d_{\bf x}^2}$. The final term in
\eqref{Hf^2} reflects the way in which the graph is embedded in $\frak{Q}$:
With $d^\|_{\bf x}$ as defined in the Theorem,
$$
\int_{[-a \pi, a \pi]^\nu}{\left|\sum_{{\bf q}\sim{\bf x}}{e^{\pm i z_\ell}}\right|^2 d {\rm Vol}_{\bf z}} = d_{\bf x} (2 a \pi)^{\nu} + {\rm cross \,\, terms},
$$
where the latter amount to
$$
2 (2 a \pi)^{\nu-2}\left((2 a \pi) d^\|_{\bf x} \int_{-a \pi}^{a\pi}{\cos(2 z) dz} + \left(\left(
{\begin{array}{c}
d_{\bf x}\\
2\\
\end{array}}
\right) - d^\|_{\bf x} \right)
\left(\int_{-a \pi}^{a\pi}{\cos(z) dz}\right)^2
\right)
$$
$$
= 4 (2 a \pi)^{\nu-2}\left((a \pi) d^\|_{\bf x} \sin(2 a \pi) + \left(\frac{d_{\bf x}^2 - d_{\bf x}}{2} - d^\|_{\bf x} \right)
2 \sin^2(a \pi)
\right).
$$
Summing all the contributions, the inequality corresponding to \eqref{KrIneq} reads
$$
(2 \pi)^\nu \sum_{j=1}^{k-1}{\lambda_j^2} \le
(2 a \pi)^\nu \left[
\left(1 - \left({\rm sinc}\left( a \pi\right)\right)^2\right) \sum_{{\bf x} \in G}{d_{\bf x}}\right.\quad\quad\quad\quad\quad\quad\quad\quad\quad\quad
$$
$$
\quad\quad\quad\quad\quad
+ \left.
\left(1 - 2 \, {\rm sinc}\left( a \pi\right)+ \left({\rm sinc}\left( a \pi\right)\right)^2\right) \sum_{{\bf x} \in G}{d_{\bf x}^2}\right.
$$
$$
\quad\quad\quad\quad\quad
+\left.
2 \left({\rm sinc}\left(2 a \pi\right) -\left({\rm sinc}\left(a \pi\right)\right)^2 \right) \sum_{{\bf x} \in G}{d^\|_{\bf x}}\right]
$$
Again letting
$a^\nu \to \kappa$, we obtain \eqref{latticeWeyl2}.
\end{proof}

The upper bound in inequality \eqref{latticeWeyl} is an increasing convex function of $\kappa=k/n$. Defining this upper bound as $S_k$, it follows that $ m_k=S_k-S_{k-1}$ is a sequence satisfying the hypotheses of Karamata's inequality.  As a consequence,

\begin{cor}
Under the same conditions as in Proposition \ref{lattice ex}, for any nondecreasing
concave function $\Phi(x)$,
$$
\sum_{j=0}^{k-1}{\Phi(\lambda_j)} \le \sum_{j=0}^{k-1}{\Phi\left(\left(1 + \frac 2 \nu\right)\frac{\pi^2 m j^{2/\nu}}{3 \,n^{1+2/\nu}}\right)},
$$
and for any
nonincreasing
convex function $\Psi(x)$,
\begin{equation}\label{partitionbd}
\sum_{j=0}^{k-1}{\Psi(\mu_j)} \ge \sum_{j=0}^{k-1}{\Psi\left(\left(1 + \frac 2 \nu\right)\frac{\pi^2 m j^{2/\nu}}{3\, n^{1+2/\nu}}\right)}.
\end{equation}
\end{cor}

The statements in the Corollary are simply applications of Karamata's Lemma \ref{Karamata}
to the bound \eqref{Weylupper}.  Lemma \ref{Karamata}
can be applied directly to \eqref{latticeWeyl} with a slight improvement, but
the form is complicated.
Interesting choices for
$\Psi$ in \eqref{partitionbd} include $x \to e^{-t x}, t \ge 0$, which corresponds to the partition function of statistical physics, and $x \to (t - x)_+^p$, which when summed on the spectrum becomes its
Riesz mean.

Two simple examples are offered to illustrate Proposition \ref{lattice ex}.

\begin{exs}
\item
1.  A path with $n$ vertices is a one-dimensional graph with eigenvalues $\displaystyle \lambda_j=4\sin^2\frac{\pi j}{2n}$.  Therefore $\displaystyle \sum_{j=1}^{k-1}\lambda_j=(2k-1)\left(1-\frac{\rm sinc(\pi\frac{2k-1}{2n})}{\rm sinc(\pi\frac{1}{2n})}\right)$, which admits the asymptotic expansion
$\displaystyle \frac{1}{k}\sum_{j=1}^{k-1}\lambda_j=\frac{\pi^2}{3}\kappa^2+O(\kappa^3)$,
thereby proving the sharpness of the bound \eqref{Weylupper} for $\nu=1$.
Considering its Laplacian energy the lower bound \eqref{LEGestimate} yields $\displaystyle LE(G)\geq \frac{4n-4}{\pi}$. A simple upper bound is given by
\begin{equation*}
  LE(G)\leq 2\sum_{j=0}^{n/2}2-4\sin^2\frac{\pi j}{2n}=\frac{2\sin\frac{\pi (n+1)}{2n}}{\sin\frac{\pi}{2n}}
\end{equation*}
(for $n$ even, but a similar expression holds for $n$ odd).
This behaves like $\displaystyle \frac{4n}{\pi}$ for $n$ large, proving the sharpness of the lower bound \eqref{LEGestimate}.
\item
2.  Consider next a cycle with $n=2n'$ vertices, which embeds minimally
in $\mathfrak{Q}^\nu$ with $\nu=2$.
The example of the cycle is an
interesting
test case for Proposition \ref{lattice ex}, because it is in a sense only slightly two-dimensional, and
because it has many different realizations in $\mathfrak{Q}^2$,
for example either with many collinear neighbors or with none.
Its eigenvalues are $\lambda_0=0$, $\displaystyle \lambda_j'=4\sin^2\frac{\pi j'}{n}$, $j'=1\ldots n'-1$ with multiplicity $2$ and $\lambda_{n-1}=\lambda_{2n'-1}=4$. We consider
the sum over an even number of eigenvalues. Let $k=2k'+1$, $k'\in\mathbb{N}$. Then $\displaystyle \sum_{j=1}^{k-1}\lambda_j=8\sum_{j'=1}^{k'}\sin^2\frac{\pi j'}{n}$. Therefore
$$
\sum_{j=1}^{k-1}\lambda_j=2k\left(1-\frac{{\rm sinc}\left(\pi \frac k n \right)}
{{\rm sinc}(\frac{\pi}{n})}\right) = 2m \kappa \left(1-\frac{{\rm sinc}\left(\pi \kappa \right)}
{{\rm sinc}(\frac{\pi}{n})}\right),
$$
which for $1 \ll k \ll n$ agrees asymptotically with the upper bound
\eqref{Weylupper} for $\nu=1$.
\end{exs}

The next application of Theorem \ref{Krlemma} makes no assumption
on $G$ other than finiteness.

\begin{cor}\label{univKnbd}
Let $G$ be a finite graph on $n$ vertices, and let $\mathfrak{M}_0$ be any set of $n(k-1)$
(ordered) pairs of vertices $\{u,v\}$.
Then for $k < n$ the eigenvalues $\lambda_k$
of the graph Laplacian $H_G$ satisfy
\begin{equation}\label{Cor6}
\sum_{j=1}^{k-1}{\lambda_j} \le \frac{1}{2n} \sum_{\left\{u,v\right\} \in \mathfrak{M}_0}{\left(d_u + d_v + 2 a_{uv}\right)}.
\end{equation}
\end{cor}

\begin{remark}
Ideally, one would optimize the choice of $\mathfrak{M}_0$, whether by favoring vertices with low values
of $d_u$ or by choosing a subset where $a_{uv} = 0$ as often as possible.  For example, if there is a large
coloring subset, choosing pairs only from it will by definition
guarantee that $a_{uv} = 0$.  The extreme case of a graph
with a large coloring subset is the star graph on $n$ vertices, and it can be verified that for such graphs,
\eqref{Cor6} becomes an equality.  Yet in the other extreme case, of the complete graph $K_n$, \eqref{Cor6} also becomes an equality.
\end{remark}

\begin{proof}
It is helpful to apply
Theorem \ref{Krlemma} thinking of the Hilbert space as the
orthogonal complement of the constant vector ${\bf 1} = \sqrt{n} \, \mathbf{\psi}^{(0)}$.
That is, $\mathcal{H}$ consists of the vectors of mean $0$.  We
take $\mathfrak{M}$ as the set of all ordered pairs
$\{u,v\}$
of vertices, the labels $u,v$ each being identified with integers $1 \dots n$,
and in this case we can simply begin the sum in
\eqref{kroeger-2} with $j=1$.
We use the counting measure on the elements of $\mathfrak{M}$ or respectively of
$\mathfrak{M}_0$, a subset of $\mathfrak{M}$ to be chosen.
For each such pair,
define the vector $b_{u,v} := {\bf e}_u - {\bf e}_v$.
As before we calculate the quantities
appearing in \eqref{kroeger-2}, beginning with
$$
\left\langle{H b_{u,v}, b_{u,v}}\right\rangle = d_u + d_v + 2 a_{uv}.
$$
(This formula is easy to see
from \eqref{weakLapDef}
by considering separately the cases where $u$ and $v$ are
connected and where they are not connected.)  For any eigenvector $\mathbf{\psi}^{(\ell)}$ other than
for $\ell=0$, the orthogonality of $\mathbf{\psi}^{(\ell)}$
to $\mathbf{\psi}^{(0)} \propto {\bf 1}$ implies that
\begin{align}\label{simplification}
\int_{\mathfrak{M}}|\langle b_{u,v},\mathbf{\psi}^{(\ell)}\rangle|^2\,d \sigma &=
\sum_{u,v=1}^n{\left(|\mathbf{\psi}^{(\ell)}_u|^2 -
2 \Re(\mathbf{\psi}^{(\ell)}_u \, \overline{\mathbf{\psi}^{(\ell)}_v})
+ |\mathbf{\psi}^{(\ell)}_v|^2\right)}\nonumber
\\
&= 2 n \|\mathbf{\psi}^{(\ell)}\| = 2 n,
\end{align}
and therefore from \eqref{KrIneq} it follows that
\begin{equation}
2 n \sum_{j=1}^{k-1}\lambda_j
\le \sum_{\left\{u,v\right\} \in \mathfrak{M}_0}{\left(d_u + d_v + 2 a_{uv}\right)},
\end{equation}
provided that the coefficient of $\lambda_k$ coming from \eqref{kroeger-2} is nonnegative,
{\it i.e.}, we require
that $0 \le 2 |\mathfrak{M}_0| - 2 n(k-1)$ (again calling on \eqref{simplification}).
This establishes Corollary \ref{univKnbd}.
\end{proof}

Next we apply the same ideas to the renormalized Laplacian:

\begin{cor}\label{univKnbdC}
Let $G$ be a finite graph on $n$ vertices, and let $\mathfrak{M}_0$ be any set of $p$ pairs of
vertices $\{u,v\}$ with $\sum_{\mathfrak{M}_0}{(d_u+d_v)} \ge 4 (k-1) m$.  Then
the eigenvalues of the renormalized Laplacian $\hat{H}_G$ satisfy
\begin{equation}\label{chungupper}
\sum_{j=1}^{k-1}{\frak{c}_j} \le  \frac{1}{4 m} \sum_{\mathfrak{M}_0}{\left(d_u + d_v + 2 a_{uv}\right)},
\end{equation}
and
\begin{equation}\label{chungupper2}
\sum_{j=1}^{k-1}{\frak{c}_j^2} \le  \frac{1}{4 m} \sum_{\mathfrak{M}_0}
{\left(d_u+d_v + 4 a_{uv} + \sum_x{\frac{1}{d_x} \left(\frac{a_{x v} d_u}{d_v} -\frac{a_{x u} d_v}{d_u}\right)^2}\right)}.
\end{equation}
\end{cor}

The final term in \eqref{chungupper2} is a measure of the deviation of $G$ from regularity.

\begin{proof}
We use Theorem \ref{Krlemma}, again choosing $\mathcal{H}$ as the
orthogonal complement of ${\bf 1}$, and
taking $\mathfrak{M}$ as the set of all pairs
$\{u,v\}$.  For each such pair, this time we
define the vector $\mathfrak{b}_{u,v} := \sqrt{d_v}{\bf e}_u - \sqrt{d_u}{\bf e}_v$.
As before we calculate the quantities on the
right side of \eqref{kroeger-2}, beginning with
\begin{align*}
\left\langle{H \mathfrak{b}_{v,w}, \ \mathfrak{b}_{v,w}}\right\rangle
&= \left\langle{H_G {\rm Deg}^{-1/2} \
\mathfrak{b}_{v,w}, {\rm Deg}^{-1/2}\ \mathfrak{b}_{v,w}}\right\rangle\\
&= \sum_{x, y}{\left(({\rm Deg}^{-1/2} \ \mathfrak{b}_{u,v})_x - ({\rm Deg}^{-1/2} \ \mathfrak{b}_{u,v})_y\right)^2}\\
&= a_{uv} \left(\sqrt{\frac{d_v}{d_u}}+\sqrt{\frac{d_u}{d_v}}\right)^2 + (d_u - a_{uv})\left(\frac{d_v}{d_u}\right)
+ (d_v - a_{uv})\left(\frac{d_u}{d_v}\right)\\
&=d_u+d_v + 2 a_{uv},
\end{align*}
just as in
Corollary \ref{univKnbd}.
This quantity is an upper bound for
$$
\sum_{j=1}^{k-1}{\frak{c}_j}\sum_{u,v}{|\left\langle{\ \mathfrak{b}_{v,w}, \mathbf{\psi}^{(j)}}\right\rangle|^2}.
$$
To evaluate the coefficient of the summand, recall that for $j>0$,
the eigenvectors $\mathbf{\psi}^{(j)}$ are orthogonal to $\mathbf{\psi}^{(0)} = {\rm Deg}^{1/2} {\bf 1}$.  Hence
$$
\sum_{u,v}{|\left\langle{\ \mathfrak{b}_{u,v}, \mathbf{\psi}^{(j)}}\right\rangle|^2} =
\sum_{u,v}{|d_v^{1/2} \mathbf{\psi}^{(j)}_u - d_u^{1/2} \mathbf{\psi}^{(j)}_v|^2} =
\sum_{u,v}{\left(d_v |\mathbf{\psi}^{(j)}_u|^2 + d_u |\mathbf{\psi}^{(j)}_v|^2\right)} =
4 m.
$$
The coefficient of $\lambda_k$
coming from this application of \eqref{kroeger-2} works out to be
$\sum_{\mathfrak{M}_0}{(d_u + d_v)} - 2 (k-1) m$.  It follows that if
this quantity is nonnegative, then
$$
\sum_{j=1}^{k-1}{{\frak c}_j} \le \frac{1}{4 m} \sum_{\mathfrak{M}_0}{\left(d_u + d_v + 2 a_{uv}\right)},
$$
as claimed.

For the sum of the squares, we instead calculate
$$
H Deg^{-1/2} \mathfrak{b}_{u,v} = \sqrt{d_u d_v} + a_{u v} \sqrt{\frac{d_u}{d_v}}+ a_{u v} \sqrt{\frac{d_v}{d_u}},
$$
from which the expectation value of $\hat{H}^2$ becomes
$$
d_u+d_v + 4 a_{uv} + \sum_x{\frac{1}{d_x} \left(\frac{a_{x v} d_u}{d_v} -\frac{a_{x u} d_v}{d_u}\right)^2},
$$
and the rest of the calculation goes as before.
\end{proof}

For the adjacency matrix, Theorem \ref{Krlemma} reduces to an elementary inequality for sums
of eigenvalues, but an inequality reflecting somewhat more of the graph structure emerges
for the sum of the $k$ smallest values of $\{\alpha_j^2\}$.  ({\it A priori} the selection
of the smallest squares is very different from the ordering
of $\{\alpha_j\}$.)

\begin{cor}\label{univKnbdA}
Let $G$ be a finite connected graph on $n$ vertices.  Then for $1 \le k < n-1$,
the eigenvalues  $\alpha_0 \ge \alpha_1 \ge \dots \ge \alpha_{n-1}$
of the adjacency matrix $A_G$ satisfy the elementary inequalities
\begin{align}\label{adj_1}
\sum_{j=0}^{n-k-1}{\alpha_j} &\ge k,\nonumber\\
\sum_{j=n-k}^{n-1}{\alpha_j} &\le -k.
\end{align}
Now let $\{\alpha_{\ell_j}\}$, $\ell = 0, \dots, n-1$ denote the eigenvalues $\alpha_j$
reordered by magnitude,
so that $|\alpha_{\ell_0}| \le |\alpha_{\ell_1}| \le \dots$.  Then for any set
$\mathfrak{M}_0$ of $n k$ ordered pairs of vertices,
\begin{equation}\label{adj_2}
\sum_{j=0}^{k-1}{\alpha_{\ell_j}^2} \le \frac{1}{2 n } \sum_{(u,v) \in \mathfrak{M}_0}{(d_u + d_v - 2 (A^2)_{u v})}.
\end{equation}
\end{cor}

(If the graph is not assumed connected, then $k$ should be replaced by $\min(k, m)$
in \eqref{adj_1}.)  We note that \eqref{adj_1} and \eqref{adj_2} become equalities for complete graphs.

\begin{proof}
The two statements in \eqref{adj_1}
are equivalent, because $\textrm{tr} \,A = 0$.  We
choose to prove the second
statement because it fits more comfortably
the schema of Theorem \ref{Krlemma}.
In this instance the Hilbert
space $\mathcal{H}$ is all of $\C^\nu$, and the set
$\mathfrak{M}$ includes all the pairs
$\{u,v\}$ and one additional element which we shall call $\omega$.
As before, for each pair of vertices we
define the vector $b_{u,v} := {\bf e}_u - {\bf e}_v$,
supplemented with the constant vector $b_{\omega} = {\bf 1}$.
For any vector $\phi$,
$|\left\langle\phi, {\bf 1}\right\rangle|^2 = |\sum_u \phi_u|^2$, and
with a calculation similar to that of the proof of Corollary \ref{univKnbd},
we find that
\begin{equation}
\sum_{u,v}{|\left\langle\phi, {\bf e}_u - {\bf e}_v \right\rangle|^2} +
2 |\left\langle\phi, {\bf 1}\right\rangle|^2 = 2 n \|\phi\|^2.
\end{equation}
We calculate that
$$
\left\langle A ({\bf e}_u - {\bf e}_v), {\bf e}_u - {\bf e}_v\right\rangle = - 2a_{uv}.
$$
and recall
$$
\left\langle A {\bf 1}, {\bf 1}\right\rangle = \sum_u{d_u} = 2 m,
$$
cf. \eqref{edgetrace}.
If $|\mathfrak{M}_0| = n k < n (n-1)$, then the quantity coming from the
left side of
\eqref{kroeger-2} vanishes, and we can conclude that
$$
n \sum_{j=n-k}^{n-1}{\alpha_j} \le -  \sum_{\mathfrak{M}_0}{a_{uv}}.
$$
For any $k < n-1$ we can in fact always find a set $\mathfrak{M}_0$ of size $n k$, and we
may furthermore preferentially include pairs $\{u,v\}$
that are connected before pairs that are not
connected.  Thus the upper bound for the sum of eigenvalues
is $-k$ (unless $k \ge m$,
which does not occur for connected graphs).  The result is
$$
\sum_{j=n-k}^{n-1}{\alpha_j} \le -k.
$$
as claimed.

Applying Theorem \ref{Krlemma} to the square of $A$ will give a bound on the sum of the $k$ smallest
values of $|\alpha_j|^2$.  We calculate that
\begin{align*}
\|A ({\bf e}_u - {\bf e}_v) \|^2 &= \sum_x{(a_{x u} - a_{x v})^2 }\\
&= \sum_x{(a_{x u} + a_{x v} - 2 a_{u x} a_{x v}) }= d_u+d_v - 2 (A^2)_{u v},
\end{align*}
and with the same condition that $|\mathfrak{M}_0| = n k < n (n-1)$, we
sum to obtain \eqref{adj_2}.
\end{proof}

\appendix
\section{Spectral analysis of the graph $G_p$.}

We consider a graph $G_p$ with $n$ vertices such that $p$ vertices, $p=1,\ldots, n-1$, are 
each the center of a star graph with $n$ vertices, and the centers
of the stars are all connected to one another.
When $p=1$, $G_p$ is a star graph. When, $p=n-1$ it is the complete graph. We introduce the following
notation: Let $I_p,0_p$ the $p\times p$ identity matrix and zero matrix, respectively. Let $J_{r,s}$ be the $r\times s$ matrix whose entries are all equal to $1$.
Let $\vec{1}_p$ and $\vec{0}_p$ be the $p$ dimensional vectors with all entries
equal to $1$ and $0$, respectively. Note that $J_{r,s}$ has the properties
that $\displaystyle J_{r,s}\vec{1}_s=s\vec{1}_r$ and $J_{r,s}J_{s,r}=sJ_{r,r}$. We
recall that $\mathbf{e}_k$, $k=1,\ldots, n$
denote the standard orthornormal basis vectors of $\mathbb{R}^n$. The positive graph
Laplacian $H_p$, the normalized Laplacian $\hat{H}_p$,
and the adjacency matrix $A_p$ of $G_p$ are given by
\begin{equation}\label{H-p}
   H_p=\left(
         \begin{array}{cc}
           pI_{n-p} & -J_{n-p,p} \\
           -J_{p,n-p} & nI_p-J_{p,p} \\
         \end{array}
       \right),
\end{equation}

\begin{equation}\label{L-p}
   \hat{H}_p=\left(
         \begin{array}{cc}
           I_{n-p} & -c_{n,p}\,J_{n-p,p} \\
           -c_{n,p}\,J_{p,n-p} & \frac{n}{n-1}\,I_p-\frac{1}{n-1}\,J_{p,p} \\
         \end{array}
       \right),
\end{equation}
where $\displaystyle c_{n,p}=p^{-\frac1{2}}(n-1)^{-\frac1{2}}$, and
\begin{equation}\label{A-p}
   A_p=\left(
         \begin{array}{cc}
           0_{n-p} & J_{n-p,p} \\
           J_{p,n-p} & J_{p,p}-I_p \\
         \end{array}
       \right).
\end{equation}

 The eigenspaces $H_p$ and the other operators can be represented as follows:
\newpage

\small{\begin{tabular}{|l|c|c|c|}
  \hline
  & Eigenvalue & Multiplicity & Eigenvectors/eigenspaces \\

  \hline
  &  &&  \\
   $H_p$ & $0$  & $1$ & $\vec{1}_n$\\
    & $p$ &$n-p-1$ & $\displaystyle \mathbf{\psi}=\sum_{k=1}^{n-p}v_k\mathbf{e}_k$ such that $\displaystyle\sum_{k=1}^{n-p}v_k=0$. \\
     &  &&  \\
      & $n$ &$p$ & $\displaystyle \mathbf{\psi}=\sum_{k=n-p+1}^{n}v_k\mathbf{e}_k$ such that $\displaystyle\sum_{k=n-p+1}^{n}v_k=0$ and \\
      &  &&  \\
      &  && $\displaystyle \mathbf{\psi}=\frac1{\sqrt{np(n-p)}}\left(
                                        \begin{array}{c}
                                          p\, \vec{1}_{n-p}\\
                                          (p-n)\,\vec{1}_p \\
                                        \end{array}
                                      \right)
       $.\\
      &  &&  \\
  \hline
   &  &&  \\
   $\hat{H}_p$ & $0$  & $1$ & $\displaystyle\frac1{\sqrt{p(2n-p-1)}}\left(
                                        \begin{array}{c}
                                          \sqrt{p}\, \vec{1}_{n-p}\\
                                          \sqrt{n-1}\,\vec{1}_p \\
                                        \end{array}
                                      \right)$\\
    &  &&  \\
    & $1$ &$n-p-1$ & $\displaystyle \mathbf{\psi}=\sum_{k=1}^{n-p}v_k\mathbf{e}_k$ such that $\displaystyle\sum_{k=1}^{n-p}v_k=0$. \\
     &  &&  \\
      & $\displaystyle \frac{n}{n-1}$ &$p-1$ & $\displaystyle \mathbf{\psi}=\sum_{k=n-p+1}^{n}v_k\mathbf{e}_k$ such that $\displaystyle\sum_{k=n-p+1}^{n}v_k=0$. \\
      &  &&  \\
      & $\displaystyle \frac{2n-1-p}{n-1}$ &$1$& $\displaystyle\frac1{\sqrt{p(2n-p-1)}}\left(
                                        \begin{array}{c}
                                          \frac{\sqrt{p(n-1)}}{\sqrt{n-p)}}\, \vec{1}_{n-p}\\
                                          -\,\frac{\sqrt{n-p}}{\sqrt{n-1)}}\,\vec{1}_p \\
                                        \end{array}
                                      \right)$.\\
      &  &&  \\
    \hline
\end{tabular}}
\small{\begin{tabular}{|l|c|c|c|}
  \hline
   &  &&  \\
   $A_p$ & $0$  & $n-p-1$ & $\displaystyle \mathbf{\psi}=\sum_{k=1}^{n-p}v_k\mathbf{e}_k$ such that $\displaystyle\sum_{k=1}^{n-p}v_k=0$.\\

      & $-1$ &$p-1$ & $\displaystyle \mathbf{\psi}=\sum_{k=n-p+1}^{n}v_k\mathbf{e}_k$ such that $\displaystyle\sum_{k=n-p+1}^{n}v_k=0$. \\
      &  &&  \\
      & $\displaystyle \rho_{+}=\frac{p-1+d_{n,p}}{2}$ &$1$& $\displaystyle\left(
                                        \begin{array}{c}
                                           \vec{1}_{n-p}\\
                                          \frac{\rho_{+}}{p}\,\vec{1}_p \\
                                        \end{array}
                                      \right)$.\\
      &  &&  \\
      & $\displaystyle \rho_{-}=\frac{p-1-d_{n,p}}{2}$ &$1$& $\displaystyle\left(
                                        \begin{array}{c}
                                           \vec{1}_{n-p}\\
                                          \frac{\rho_{-}}{p}\,\vec{1}_p \\
                                        \end{array}
                                      \right)$.\\
      &  &&  \\
  \hline
\end{tabular}}
\\

Here $\displaystyle d_{n,p}=\sqrt{(p-1)^2+4p(n-p)}$ and the eigenvectors for $A_p$ corresponding to $\displaystyle \rho_{\pm}$ are not normalized.

{\bf Acknowledgments}  We wish to acknowledge the assistance of Mr. Thomas Boutin
for numerical studies of some of the inequalities reported here.
E.H. is also grateful to the \'Ecole Polytechnique F\'ed\'erale de Lausanne for
hospitality that supported this collaboration.


\end{document}